\documentclass{article}

\usepackage{arxiv}

\usepackage[round]{natbib}

\usepackage{graphicx}
\usepackage{amssymb}
\usepackage{amsthm}

\newtheorem{theorem}{Theorem}
\newtheorem{corollary}{Corollary}
\newtheorem{lemma}[theorem]{Lemma}
\newtheorem{conjecture}{Conjecture}

\usepackage{amsmath}
\usepackage{float}

\title{Proof of The TAP Free Energy for High-Dimensional Linear Regression with Spherical Priors at All Temperatures} 

\author{Zhiyuan Yu \\
  Department of Statistics\\
  University of Illinois Urbana-Champaign \\
  Champaign, IL 61820 \\
  \texttt{yu124@illinois.edu}\\
  \And
  Jingbo Liu \\
  Department of Statistics\\
  University of Illinois Urbana-Champaign \\
  Champaign, IL 61820 \\
  \texttt{jingbol@illinois.edu}
}
\renewcommand{\shorttitle}{Proof of the Tap Free Energy}

\begin{document}
\maketitle

\begin{abstract}

Approximate inference is central to Bayesian learning, with variational inference (VI) providing a scalable framework for posterior approximation. While mean-field VI often fails in high dimensions, the more refined Bethe approximation, equivalent to the Thouless-Anderson-Palmer (TAP) free energy in statistical physics, has long been conjectured to capture Bayes-optimal behavior. We prove that the TAP formula holds for Bayesian linear regression with a uniform spherical prior at all noise levels ($\Delta>0$), extending the result of Qiu and Sen (2023) in the high-noise regime. Our argument constructs a ridge regression functional that dominates the TAP free energy, yielding the first rigorous analysis of the global optimizer of the non-concave TAP functional for a planted inference model at an arbitrary noise level. This verifies that TAP, rather than mean-field, is the correct variational description in this setting.
\end{abstract}

\section{INTRODUCTION}


Variational inference (VI) has become a central tool in statistics and machine learning for approximating intractable posteriors \citep{wainwright2008graphical}. While classical VI is usually formulated in terms of mean-field approximations,recent statistical results indicate that naive mean-field VI can be asymptotically consistent only under restrictive conditions \citep{mukherjee2021variationalinferencehighdimensionallinear}, 
and may fail in genuinely high-dimensional regimes, motivating the study of refined free energies. A more general principle is given by the Bethe free energy, originally developed in statistical physics \citep{bethe1935statistical}. The Bethe free energy is a variational functional whose stationary points coincide with the fixed points of the belief propagation algorithm \citep{yedidia2001bethe,yedidia2003understanding}. Recent work has provided rigorous analyses of the convergence of belief propagation to Bethe free energy stationary points \citep{koehler2019fast}. In spin glass theory, the corresponding variational principle is the Thouless–Anderson–Palmer (TAP) free energy \citep{thouless1977solution}. Their connection has also been established by \citet{welling2013belief}, who demonstrate that in binary networks the Bethe free energy agrees with the TAP free energy up to second order in the weights. In high-dimensional inference, \citet{Krzakala2014} formulated the functional under the name Bethe free energy, while \citet{qiu2023tap} analyzed the same functional under the TAP terminology in their rigorous study of linear regression.

Recent progress has advanced the theoretical understanding of the TAP free energy. \citet{qiu2023tap} derived asymptotic formulas in certain high-temperature regimes, while \cite{celentano2023meanfieldvariationalinferencetap} analyzed the local geometry of the TAP functional, establishing that the fixed points of Approximate Message Passing (AMP) correspond to approximate local maximizers. Nevertheless, these results do not provide a complete characterization of the variational landscape. 
This naturally leads to the central question:
\begin{quote}
\emph{How to characterize the global maximizers of the TAP free energy?}
\end{quote}

Addressing this question is not only of intrinsic interest in statistical physics and Bayesian inference, but also has implications for modern machine learning. 
In particular, whether the global optimizer of the TAP free energy correctly characterizes posterior means is an important question, especially due to its implications for score-based generative modeling \citep{song2021score,montanari2023posterior,cui2024sampling}.
Thus, clarifying the behavior of TAP maximizers provides both a rigorous foundation for classical inference methods and contemporary relevance for generative modeling.

The Thouless--Anderson--Palmer (TAP) free energy provides a refined variational principle that incorporates Onsager correction terms \citep{thouless1977solution,mezard1987spin}. 
Unlike purely algorithmic approaches, TAP defines an explicit objective that can be optimized by gradient-based methods, which is often regarded as providing greater robustness compared to AMP \citep{el2022sampling}. 
From a statistical perspective, global maximizers of TAP correspond to centers of high-probability states, thereby offering a geometric characterization of posterior clustering \citep{subag2017geometry,mezard1987spin,fan2021tap}. 
From a mathematical perspective, TAP free energies have become fundamental objects of study in probability theory, leading to rigorous formulations in spin glass models \citep{chen2020generalizedtapfreeenergy,subag2023tap} and underpinning major advances such as the resolution of the Ising perceptron capacity problem \citep{huang2024capacity} 
and Talagrand’s operator norm conjecture \citep{brennecke2023operator}. 
These developments demonstrate the broad significance of TAP as a unifying framework in statistical physics, Bayesian inference, and modern machine learning.

Despite these advances, in Bayesian linear regression an essential step remains to be established: whether the global maximum of the TAP functional converges to the true free energy. Constructing this equality would be an essential step toward showing that the global maximizer coincides with the posterior mean. 
Fixed point calculations of  \citep{Krzakala2014} suggest a positive answer (see Conjecture~\ref{conj1} below), but a rigorous proof is difficult because the TAP functional is non-concave and may admit spurious stationary points.

A first step towards rigorously addressing this open problem was taken by \citet{qiu2023tap},
who focused on the high-temperature regime $\Delta>\Delta_0$ (for a certain threshold $\Delta_0$) under uniform spherical priors. 
Consider a Gaussian random linear estimation framework, with data points ${(y_i, x_i) : 1 \leq i \leq n}$, where $\mathbf{y}=(y_i)_{i=1}^N \in \mathbb{R}^N$ are responses and $\mathbf{x}=(x_{ij})_{i=1,j=1}^{N,p} \in \mathbb{R}^{N\times p}$ are feature vectors. The model assumes: 
\begin{align}\label{model1}
\mathbf{y}=\mathbf{x}\boldsymbol{\beta}+\sqrt{\Delta}\mathbf{z},
\end{align}
where $\boldsymbol{\beta} \in \mathbb{R}^p$ is the coefficient vector and $\mathbf{z}=(z_i)_{i=1}^N \sim \mathcal{N}(0, I_N)$ are independent and identically distributed errors. Using a prior distribution $\pi$ for $\boldsymbol{\beta}$, the posterior distribution is constructed as follows:
\begin{align}\label{e2}
\frac{\mathrm{d}P_{\boldsymbol{\beta}|\mathbf{y,x}}}{\mathrm{d}\pi}(\boldsymbol{\beta}|\mathbf{y,x})
=\frac{1}{\mathcal{Z}_p}e^{-\frac{1}{2\Delta}\|\mathbf{y}-\mathbf{x}\boldsymbol{\beta}\|^2}.
\end{align}
The posterior normalization $\mathcal{Z}_p$ is also called the partition function in statistical physics. We also refer to \begin{align}\label{FE}
    F_p=\frac{1}{p}\ln(\mathcal{Z}_p)
\end{align} as the free energy. 
While \citet{qiu2023tap} showed that the TAP approximation correctly characterizes $F_p$ when $\pi$ is a uniform spherical distribution and $\Delta>\Delta_0$, extending its validity to the low-temperature regime has remained an open question:

\begin{conjecture}\label{conj1}{\rm{\citep{Krzakala2014,qiu2023tap}}} 
For Bayesian linear regression with a uniform prior on the sphere $\mathbb{S}^{p-1}(\sqrt{p})$, the TAP free energy provides an asymptotically exact characterization of the true free energy in the high-dimensional limit, for all noise levels $\Delta > 0$. That is,
\begin{align}\label{TAP}
 \lim_{p \to \infty} \frac{1}{p} \log \mathcal{Z}_p=&\lim_{p \to \infty} \max_{\|a\| \leq \sqrt{p}}\{ -\frac{1}{2\Delta p}\|\mathbf{y}-\mathbf{x}a\|^2\notag\\&-\frac{\alpha}{2}\ln(1+\frac{1-\|a\|^2/p}{\Delta\alpha })+\frac{1}{2}\ln(1-\frac{\|a\|^2}{p})\}.
\end{align}
\end{conjecture}

Throughout this paper, we refer to the quantity in \eqref{TAP} as the TAP free energy, which corresponds to the log-partition function or the free entropy in the statistical physics literature, up to a sign.

The spherical prior assumption enables the geometric arguments of \citet{qiu2023tap}, while
a similar analysis for spherical spin models was previously provided by \citet{belius2022high}. 
In addition, variants of spherical or normalized-signal models arise in applied contexts such as signal processing and compressed sensing on spherical domains \citep{cornelius2016compressed} underscoring the relevance of spherical assumptions beyond purely theoretical considerations.
In the spherical case, the posterior distribution is concentrated in a sphere cap,
and the high-temperature assumption is necessary for the arguments of \cite{qiu2023tap} in justifying that the free energy in \emph{any} sphere cap is that of a pure state.
It is worth noting that for spherical Sherrington–Kirkpatrick(SK) models, the TAP expression was established for any temperature, 
if the TAP optimization is constrained to satisfy the so-called Plefka condition to ensure high-temperature phase of the sphere cap \citep{belius2022high}, but it is not obvious whether a similar condition is needed for the Bayes regression problem since $\bf y$ is a ``planted'' signal.
If $\Delta<\Delta_0$, then some large sphere caps are not in a pure state,
but these large spheres may not be relevant for the maximizers of the TAP free energy, 
hence the validity of the TAP free energy representation is elusive.
A more recent work \cite{celentano2023meanfieldvariationalinferencetap} established a version of Conjecture~\ref{conj1} for i.i.d.\ priors, by showing that there exists a local maximum value of the TAP free energy that converges to the free energy.
This does not establish Conjecture~\ref{conj1}, because there may exist local maximizers of the TAP free energy that are not global. 
Indeed, we provide an example in Appendix~\ref{app_A} where the TAP function is non-concave, so that spurious local maximizers is possible.

In this work, 
we prove that the TAP approximation is valid for all $\Delta>0$ under uniform spherical priors, establishing Conjecture~\ref{conj1}. 
As noted before, the key challenge in establishing Conjecture~\ref{conj1} is to exclude the possibility of a larger global maximum value.
While the TAP free energy functional may not be globally concave,
we choose a surrogate functional based on ridge regression that can be shown to globally dominate it.
The surrogate functional is globally concave, 
so that its global maximum can be easily controlled.
Another key ingredient of our proof is the replica symmetric (RS) free energy expression established by \cite{reeves2019replica} and \cite{Barbier2020}.
Our result shows that, unlike in \cite{belius2022high}, 
there is no need to impose the Plefka condition that lower bounds $\|a\|$, for TAP optimization in the Bayes-optimal setting ($\bf y$ is generated by signal from a matching prior),
which aligns with the current practice.
Our result also implies that any near global maximizer of the TAP free energy must be close to the solution of a certain ridge regression.

The remainder of this paper is structured as follows: Section 2 introduces the model setup, notations, and defines the TAP free energy in the context of spherical priors. Section 3 presents our main result, which establishes the asymptotic equivalence between the TAP free energy and the true free energy. Section 4 provides a detailed proof, which proceeds by establishing a Gaussian-spherical equivalence, analyzing the ridge regression surrogate, and connecting it to the RS potential. Section~5, presents numerical simulations that empirically investigate the geometry of the TAP free energy landscape. Section 6 concludes with a discussion of implications and directions for future work.

\section{PRELIMINARIES}

We assume that $\pi$ is a uniform prior on $S_{p-1}(\sqrt{p}) := \left\{ \beta \in \mathbb{R}^p : \|\beta\| = \sqrt{p} \right\}$. The matrix $\mathbf{x}$ has i.i.d.\ Gaussian entries $x_{i,j}\sim \mathcal{N}(0, \frac{1}{N})$, so that its operator norm is $O(1)$. The measurement rate $N/p$ is denoted by $\alpha\in(0,\infty)$, which is assumed to be a constant.
Therefore, the model is characterized by two parameters, $\Delta$ and $\alpha$.
A random variable converging to zero in probability is denoted by $o_P(1)$.

\subsection{TAP Free Energy}

Under the aforementioned settings, let $a$ be a vector in $S_{p-1}$ with the norm $\|a\|^2=pq$. 
We can then define the following TAP function for
$a$, which was derived by 
\citet{qiu2023tap}:
\begin{align}\label{tap_equation}
    f_{\mathrm{TAP}}(a)=&-\frac{1}{2\Delta p}\|\mathbf{y}-\mathbf{x}a\|^2-\frac{\alpha}{2}\ln(1+\frac{1-q}{\Delta\alpha})\notag\\&+\frac{1}{2}\ln(1-q).
\end{align}

    The TAP free energy is subsequently obtained by evaluating the supremum of $f_{\mathrm{TAP}}$, or equivalently 
    \begin{align}
        F_{TAP}=\sup_{\|a\|^2 \le p}f_{\mathrm{TAP}}(a).
    \end{align}
In fact, \citet{Krzakala2014} derived a general version of \eqref{tap_equation} for i.i.d.\ priors (as we will see, spherical prior is asymptotically equivalent to i.i.d.\ Gaussian prior).
They showed that  the stationary points of $f_{\mathrm{TAP}}$ coincides with the fixed points of the AMP algorithm, 
suggesting the conjecture that $F_{TAP}$ is asymptotically equivalent to the true free energy.
\citet{qiu2023tap} established that $F_{TAP}-F_p$ converges in probability to $0$ as $p \to \infty$, provided that $\Delta$ is sufficiently large. And in this work, we will extend their result by proving that this convergence holds for any $\Delta > 0$, thereby removing the restriction on $\Delta$.

\subsection{RS Formula}
We mostly follow the presentation of \citet{Barbier2020},
with some adaptations of their notations to be consistent with \eqref{model1}.
Consider the model in equation \eqref{model1}, where the prior distribution of $\boldsymbol{\beta}$ is given by $P_0^{\otimes p}$.
We define \begin{align}
    \Sigma(E; \Delta)^{-2} := \frac{\alpha}{\alpha \Delta + E} .
\end{align} Let $ i(\tilde{S}; \tilde{Y}) $ denote the mutual information (MI) for the denoising model $ \tilde{y} = \tilde{s} + \tilde{z} \Sigma$, where $ \tilde{s} \sim P_0 $ and $ \tilde{z} \sim \mathcal{N}(0, 1)$. The corresponding expression for $i(\tilde{S}; \tilde{Y})$ is given by:
\begin{align}\label{MI_denoising_model}
    &i(\tilde{S};\tilde{Y})\notag\\=&-\mathbb{E}_{\tilde{S},\tilde{Z}}\left[\ln\mathbb{E}_{\tilde{\beta}}\left[\exp\left(-\frac{(\tilde{\beta}-(\tilde{S}+\tilde{Z}\Sigma))^2}{2\Sigma^2}\right)\right]\right]-\frac{1}{2}.
\end{align}

The replica method then yields the RS potential for model \eqref{model1} with prior $P_0^{\otimes p}$:
\begin{align}\label{RS_potential}
    i^{\mathrm{RS}}(E;\Delta) := i(\tilde{S};\tilde{Y})+\frac{1}{2}(\alpha \ln(1+E/(\alpha\Delta))-\frac{E}{\Sigma^2}).
\end{align}
This formula was first derived by \citet{Tanaka2002} and later rigorously proved by \citet{reeves2019replica}. \citet{Barbier2020} subsequently provided a different proof,
 assuming that the prior is discrete and $i^{\mathrm{RS}}$ has at most three stationary points as a function of $E$. Reeves and Pfister also explored this relationship under the assumption that the prior has finite fourth moment and a "Single-Crossing Property". Additionally, as noted by \citet[Remark~3.6]{Barbier2020}, the discrete prior can be extended to more generalized distributions, where the constraint is relaxed as discussed in the work of \citet{barbier2019optimal}. Specifically, under the assumptions, the following holds:
 \begin{lemma}{\rm{\citep{reeves2019replica}}}
Let $i^{\mathrm{RS}}(E; \Delta)$ be the RS potential defined in \eqref{RS_potential}, and let $\mathcal{Z}_p$ be the partition function defined in \eqref{e2}. Under the assumption that $P_0$ has finite fourth moment and satisfies the single-crossing property, we have:
\begin{align}
    &-\frac{\alpha}{2}-\min_{E\ge 0} i^{\mathrm{RS}}(E;\Delta) = \lim_{p\to \infty}\{\frac{1}{p}\mathbb{E}\ln\mathcal{Z}_p\}
    \label{e_f}
\end{align}
\end{lemma}
 Thus, if we can establish that our prior satisfies the RS formula and $\min_{E\ge 0} i^{\mathrm{RS}}(E;\Delta) = -\max_{\|a\|^2\le p} f_{\mathrm{TAP}}(a)-\frac{\alpha}{2}$ holds, it would confirm the correctness of the TAP free energy in the uniform spherical setting for all $\Delta > 0$. And we defer the verification of the RS assumptions to Appendix~\ref{section_B}. 
 
\section{MAIN RESULT}

In this section, we formally state our main theoretical result, which establishes the asymptotic correctness of the TAP free energy approximation for high-dimensional Bayesian linear regression under a uniform spherical prior. 

 \begin{theorem}\label{theorem_1}
Fix $ \alpha > 0 $ and $ \Delta > 0 $. Consider the Bayesian linear regression model
\begin{align}
y = \mathbf{x} \beta + \sqrt{\Delta} \, \epsilon,\label{mod_1}
\end{align}
where $ \mathbf{x} \in \mathbb{R}^{N \times p} $ has i.i.d.\ $ \mathcal{N}(0, \tfrac{1}{N}) $ entries, $ \epsilon \sim \mathcal{N}(0, I_N) $, and the prior on $ \beta \in \mathbb{R}^p $ is uniform on the sphere $ \mathbb{S}^{p-1}(\sqrt{p}) $. Assume that $ N/p \to \alpha $ as $ p \to \infty $.

Then, as $ p \to \infty $, we have convergence in probability (with respect to the randomness in $ \mathbf{x} $, $ \beta $, and $ \epsilon $):
\begin{align}\label{theorem}
\lim_{p \to \infty} \max_{\|a\| \leq \sqrt{p}} f_{\mathrm{TAP}}(a) = \lim_{p \to \infty} F_p,
\end{align}
where $ f_{\mathrm{TAP}}(a) $ denotes the TAP free energy functional defined in \eqref{tap_equation} and $ F_p = \frac{1}{p} \ln \mathcal{Z}_p $ is the free energy shown in \eqref{FE}.
\end{theorem}

 \begin{proof}
 The result follows from Lemma~\ref{lower} and Lemma~\ref{upper} ahead.
 \end{proof}
As a corollary, any near maximizer of the TAP functional must be close to a ridge regressor:
\begin{corollary}\label{cor1}
In the setting of Theorem~\ref{theorem_1},
define 
$R_{p,r}:=\max\{\frac1{p}\|a-\hat{a}_{\rm ridge}\|_2^2\colon
f_{\rm TAP}(a)\ge
\sup_{\|a\|\le \sqrt{p}}f_{\rm TAP}(a)
-r\}$,
for any $r\ge 0$,
where $\hat{a}_{\rm ridge}:={\rm argmin}_{a\in\mathbb{R}^p}\{\frac{1}{\Delta}\|\mathbf{y}-\mathbf{x}a\|^2+\|a\|^2\}$.
Then we have the convergence in probability $\lim_{p\to\infty} R_{p,r}\le 2r$.
\end{corollary}
\begin{proof}
Suppose that $a$ satisfies $f_{\rm TAP}(a)\ge \sup_{\|a\|\le \sqrt{p}}f_{\rm TAP}(a)-r$.
Then for $C$ defined in \eqref{e_c},
we have 
$-\frac{1}{2\Delta p}\|\mathbf{y}-\mathbf{x}a\|^2-\frac1{2p}\|a\|^2+C
\ge f_{\rm TAP}(a)\ge\sup_{\|a\|\le \sqrt{p}}f_{\rm TAP}(a)-r\ge \lim_{p\to\infty}F_p-r+o_p(1)$.
On the other hand, in the proof of Lemma~\ref{upper} we obtained $\lim_{p\to\infty}\sup_{a\in\mathbb{R}^p}
\{-\frac{1}{2\Delta p}\|\mathbf{y}-\mathbf{x}a\|^2-\frac1{2p}\|a\|^2\}+C\le \lim_{p\to\infty}F_p$.
Since $\nabla_a^2\{-\frac{1}{2\Delta p}\|\mathbf{y}-\mathbf{x}a\|^2-\frac1{2p}\|a\|^2\}\preceq \nabla_a^2\{-\frac1{2p}\|a\|^2\}= -\frac1{p}I_p$ (in terms of the ordering of positive semidefinite matrices),
we see that $\frac1{2p}\|a-\hat{a}_{\rm ridge}\|^2\le r+o_p(1)$,
and the result follows.
\end{proof}

\section{PROOF OF THE MAIN RESULT}

\subsection{Equivalence of i.i.d.\ Normal Prior and the Spherical Prior}
In this section, we show that the asymptotic free energy under the spherical prior can be calculated using the result of \citet{reeves2019replica} with i.i.d.\ Gaussian prior.
Intuitively this is simply because a high-dimensional gaussian vector is concentrated near a sphere,
but the actual proof also relies strongly on the Bayes-optimal assumption; 
we expect that the result will no longer be true in a mismatched setting (misspecified model) where the true prior differs from the prior used in defining the TAP expression.

Let 
$\tilde{\pi}:=\mathcal{N}(0,I_p)$, 
and 
\begin{align}
\mathbf{y}:=\mathbf{x}\boldsymbol{\beta}_0+\sqrt{\Delta}\mathbf{z}
\label{e_11}
\end{align}
where $\boldsymbol{\beta}_0\sim \tilde{\pi}$, 
${\bf x}\sim \mathcal{N}(0,\frac1{N}I_N\otimes I_p)$
and 
${\bf z}\sim \mathcal{N}(0,I_N)$.
Let $\Pi$ denote the projection onto the centered sphere of radius $\sqrt{p}$.
Note that $\Pi(\boldsymbol{\beta}_0)$ thus follows the uniform distribution on the centered sphere of radius $\sqrt{p}$,
which we denote by $\pi$.
Set 
\begin{align}
\mathbf{y}':=\mathbf{x}\Pi(\boldsymbol{\beta}_0)+\sqrt{\Delta}\mathbf{z}.
\label{e12}
\end{align}
Recall that $o_P(1)$ denotes a random variable vanishing in probability, and similarly for $o_P(\sqrt{p})$ etc.

\begin{lemma}
Consider random $\bf x$, $\bf y$, and $\bf y'$ defined in \eqref{e_11} and \eqref{e12}.
We have 
\begin{align}
&\frac1{p}\ln\int 
e^{-\frac1{2\Delta}\|{\bf y-x}\boldsymbol{\beta}\|^2}
d\tilde{\pi}(\boldsymbol{\beta})
\notag\\=&\frac1{p}\ln\int
e^{-\frac1{2\Delta}\|{\bf y'-x}\boldsymbol{\beta}\|^2}
d\pi(\boldsymbol{\beta}) +
o_P(1).
\end{align}
\end{lemma}
Note that as a consequence of this lemma and the result of \citet{reeves2019replica}, the right side of \eqref{theorem} (for spherical prior) can be calculated using \eqref{RS_potential} and \eqref{e_f} with the Gaussian prior.

\begin{proof}

Let $(u_p)_{p\ge 1}$ be a nonnegative vanishing sequence satisfying \begin{align}
\lim_{p\to \infty}
\mathbb{P}\left[\left|\frac{\|\boldsymbol{\beta}_0\|}{\sqrt{p}}-1\right|>u_p\right]=0,
\label{e11}
\end{align}
where $\boldsymbol{\beta}_0\sim \tilde{\pi}$,
which is possible if $u_p=\omega(\frac1{\sqrt{p}})$, as established by \citet[Section~3.1]{vershynin2010introduction}.

We introduce an auxiliary probability measure $\tilde{P}$, which will serve as a convenient reference distribution in the subsequent analysis.
Let $\tilde{P}_{\boldsymbol{\beta}_0 {\bf xz}}=\tilde{\pi}\times \mathcal{N}(0,\frac1{N}I_{N\times p})\times \mathcal{N}(0,I_N)$, 
let $\mathbf{y}$ be a function of $({\bf x ,z},\boldsymbol{\beta}_0)$ as define by \eqref{e_11},
and then define $\tilde{P}_{\boldsymbol{\beta}_0 {\bf xzy}\hat{\beta}}
=
\tilde{P}_{\boldsymbol{\beta}_0 {\bf xzy}}
\tilde{P}_{\boldsymbol{\hat{\beta}} |{\bf xy}}$
by setting $\tilde{P}_{\boldsymbol{\hat{\beta}} |{\bf xy}}
=
\tilde{P}_{\boldsymbol{\beta}_0 |{\bf xy}}$.
In other words, under $\tilde P$, the estimator $\hat{\boldsymbol{\beta}}$ is conditionally distributed as an independent posterior sample of the planted parameter $\boldsymbol{\beta}_0$ given the data $({\bf x},\mathbf y)$.
Hence $\boldsymbol{\hat{\beta}}$ is a posterior sample following the Bayesian rule.
By symmetry, we have the equivalence of the marginal distributions $\tilde{P}_{\boldsymbol{\hat{\beta}}}=\tilde{P}_{\boldsymbol{\beta}_0}$.
Consequently, if we define 
\begin{align}
\lambda_{\boldsymbol{\beta}_0{\bf xz}}
:=
\mathbb{P}\left[\left.\left|\frac{\|\boldsymbol{\hat{\beta}}\|}{\sqrt{p}}-1\right|>u_p\right|\boldsymbol{\beta}_0, {\bf x,z}\right]
\end{align}
under $\tilde{P}$, then 
\begin{align}
&\mathbb{E}[\lambda_{\boldsymbol{\beta}_0{\bf xz}}]
=
\mathbb{P}\left[\left|\frac{\|\boldsymbol{\hat{\beta}}\|}{\sqrt{p}}-1\right|>u_p\right]
\notag\\=&\mathbb{P}\left[\left|\frac{\|\boldsymbol{\beta}_0\|}{\sqrt{p}}-1\right|>u_p\right]
=o(1).
\label{e17}
\end{align}
Now define 
\begin{align}
G:=\left\{\boldsymbol{\beta}\colon 
\left|\frac{\|\boldsymbol{\beta}\|}{\sqrt{p}}-1\right|
\le u_p
\right\}.
\end{align}
Then
\begin{align}
&\frac1{p}\ln\int 
e^{-\frac1{2\Delta}\|{\bf y-x}\boldsymbol{\beta}\|^2}
d\tilde{\pi}(\boldsymbol{\beta})
\notag\\=&\frac1{p}\ln\int_G 
e^{-\frac1{2\Delta}\|{\bf y-x}\boldsymbol{\beta}\|^2}
d\tilde{\pi}(\boldsymbol{\beta})
-
\frac1{p}\ln \tilde{P}_{\boldsymbol{\beta}_0|{\bf xy}}[G|{\bf xy}]
\notag\\
=&\frac1{p}\ln\int_G
e^{-\frac1{2\Delta}\|{\bf y-x}\boldsymbol{\beta}\|^2}
d\tilde{\pi}(\boldsymbol{\beta})
+
o_P(1),
\label{e18}
\end{align}
where the last step follows 
by $\tilde{P}_{\boldsymbol{\beta}_0|{\bf xy}}[G|{\bf xy}]=1-\lambda_{\boldsymbol{\beta}_0{\bf xz}}$ and \eqref{e17}.
Now let $\mu$ be the distribution of $\boldsymbol{\beta}_0\sim \tilde{\pi}$ conditioned on  $\boldsymbol{\beta}_0\in G$.
We have from \eqref{e11} that 
\begin{align}
&\frac1{p}\ln\int_G
e^{-\frac1{2\Delta}\|{\bf y-x}\boldsymbol{\beta}\|^2}
d\tilde{\pi}(\boldsymbol{\beta})
+
o_P(1)
\notag\\=&
\frac1{p}\ln\int
e^{-\frac1{2\Delta}\|{\bf y-x}\boldsymbol{\beta}\|^2}
d\mu(\boldsymbol{\beta})
+
o_P(1).
\label{e19}
\end{align}
Recall that $\Pi$ is the projection of $\boldsymbol{\beta}$ onto the centered sphere of radius $\sqrt{p}$.
For any $\boldsymbol{\beta}_0\in G$,
we have 
$\|\boldsymbol{\beta}-\Pi(\boldsymbol{\beta})\|=o(\sqrt{p})$,
and hence
\begin{align}
&\left|
\|{\bf y-x}\boldsymbol{\beta}\|^2
-
\|{\bf y-x}\Pi(\boldsymbol{\beta})\|^2
\right|\notag\\=&
|\left<{\bf x}(\boldsymbol{\beta}-\Pi(\boldsymbol{\beta}))
,
2{\bf y-x}
(\boldsymbol{\beta}+\Pi(\boldsymbol{\beta})
\right>|
\\
=&o(\sqrt{p}(\sigma_{\max}({\bf x})\|y\|
+\sigma_{\max}({\bf x})^2\sqrt{p}))
\\
=&o_P(p),
\end{align}
where we used Lemma~\ref{lem3}.
Therefore, we obtain
\begin{align}
&\frac1{p}\ln\int
e^{-\frac1{2\Delta}\|{\bf y-x}\boldsymbol{\beta}\|^2}
d\mu(\boldsymbol{\beta})
+
o_P(1)
\notag\\=&
\frac1{p}\ln\int
e^{-\frac1{2\Delta}\|{\bf y-x}\Pi(\boldsymbol{\beta})\|^2}
d\mu(\boldsymbol{\beta})
+
o_P(1)
\notag\\
=&\frac1{p}\ln\int
e^{-\frac1{2\Delta}\|{\bf y-x}\boldsymbol{\beta}\|^2}
d\pi(\boldsymbol{\beta})
+
o_P(1).
\label{e24}
\end{align}
Then
\begin{align}
&\|{\bf y-x}\boldsymbol{\beta}\|^2
-
\|{\bf y'-x}\boldsymbol{\beta}\|^2
\notag\\=&\left<{\bf y-y'},
{\bf y+y'-2 x}\boldsymbol{\beta}
\right>
\\
=&
\left<{\bf x}(\boldsymbol{\beta}_0-\Pi(\boldsymbol{\beta}_0)),
{\bf y+y'-2 x}\boldsymbol{\beta}
\right>
\\
=&o_P(p),
\end{align}
since $\|\boldsymbol{\beta}_0-\Pi(\boldsymbol{\beta}_0)\|=o_P(\sqrt{p})$ when $\boldsymbol{\beta}_0\sim \tilde{\pi}$, and $\|{\bf y+y'-2 x}\boldsymbol{\beta}\|=O_P(\sqrt{p})$.
In turn, 
\begin{align}
&\frac1{p}\ln\int
e^{-\frac1{2\Delta}\|{\bf y-x}\boldsymbol{\beta}\|^2}
d\pi(\boldsymbol{\beta})
+
o_P(1)
\notag\\=
&\frac1{p}\ln\int
e^{-\frac1{2\Delta}\|{\bf y'-x}\boldsymbol{\beta}\|^2}
d\pi(\boldsymbol{\beta})
+
o_P(1).
\end{align}
\end{proof}

\begin{lemma}{\rm{\citep{vershynin2010introduction}}}
\label{lem3}
Let $A$ be a $N\times p$ matrix with i.i.d. standard normal entries. For every $k>0$, the largest eigenvalue of $A$ satisfies:
\begin{align}
    \sigma_{max}(A)\le \sqrt{N}+\sqrt{p}+k,
\end{align}
with probability at least $1-2\exp(-k^2/2)$.
\end{lemma}

\subsection{TAP Free Energy at a Ridge Solution}
Our proof relies on analyzing a ridge regression solution:
\begin{align}
a_{\Delta}=A_{\Delta}^{-1}\mathbf{x}^{\top}\mathbf{x}\boldsymbol{\beta}+A_{\Delta}^{-1}\mathbf{x}^{\top}\epsilon,
\label{e_ridge}
\end{align}
where $A_{\Delta}:=\mathbf{x}^{\top}\mathbf{x}+\Delta I$. 
To analyze the asymptotic behavior of the ridge regression solution, we consider the normalized squared norm:
\begin{align}
\lim_{p\to\infty}\frac{\|a_{\Delta}\|^2}{p}
&=\lim_{p\to\infty}\frac{1}{p}
(\|A_{\Delta}^{-1}\mathbf{x}^{\top}\mathbf{x}\boldsymbol{\beta}\|^2
+\|A_{\Delta}^{-1}\mathbf{x}^{\top}\epsilon\|^2)
\\
&=\lim_{p\to\infty}\frac{1}{p}(
{\rm tr}(\mathbf{x}^{\top}\mathbf{x}A_{\Delta}^{-2}\mathbf{x}^{\top}\mathbf{x})
+
\Delta{\rm tr}(\mathbf{x}A_{\Delta}^{-2}\mathbf{x}^{\top}))\\
&=1+\lim_{p\to\infty}-\frac{\Delta}{p}{\rm tr} ( A_{\Delta}^{-1} ),\label{q_tr}
\end{align}
where the second equation can be derived using the independence of $\epsilon$ and $\boldsymbol{\beta}$ as well as LLN. 

The trace of $A_{\Delta}^{-1}$ in the preceding equations can be determined using the Marchenko–Pastur distribution, which provides the asymptotic behavior of these quantities for large matrices. As the the eigenvalues of $\mathbf{x}^{\top}\mathbf{x}$ follow Marchenko–Pastur law, the Stieltjes transform yields the following:
\begin{align}
&\lim_{p\to\infty}\frac1{p}{\rm tr} ( A_t^{-1} ) =\lim_{p\to\infty}\frac1{t}T(t),\label{MP_law}\\
    &T(t) := \frac{-\alpha+1-t\alpha+\sqrt{(-\alpha+1-t\alpha)^2+4\alpha t}}{2}.
\end{align}
The complete statements and proof can be found in the work of \citet{Bai2010}. 

Next we shall establish the following result:
\begin{lemma}\label{lem4}
Define $q_\Delta:=1-T(\Delta)$. Then $\lim_{p\to\infty}\mathbb{P}[\|a_{\Delta}\|< \sqrt{p}]=1$. Moreover, under the convention that $f_{\mathrm{TAP}}(a) = -\infty$ for any $a$ with $\|a\| > \sqrt{p}$ (i.e., outside the domain of the TAP free energy functional), we have convergence in probability:
\begin{align}
&\lim_{p\to\infty}f_{\mathrm{TAP}}(a_{\Delta})\notag\\=&-\frac{1}{2}\alpha+\frac{1}{2}q_\Delta-\frac{\alpha}{2}\ln(1+\frac{1-q_\Delta}{\Delta\alpha})+\frac{1}{2}\ln(1-q_\Delta).
\end{align}
\end{lemma}

\begin{proof}
    From \eqref{q_tr} and \eqref{MP_law}, it remains to analyze the asymptotic behavior of the quantity $\frac{1}{p}\|\mathbf{y}-\mathbf{x}a_\Delta\|^2$. Exploiting the independence between $\epsilon$ and $\beta$, we obtain:
\begin{align}\label{first_term}  &\lim_{p\to\infty}\frac{1}{p}\|\mathbf{y}-\mathbf{x}a_\Delta\|^2\notag\\=&\lim_{p\to\infty}\frac{1}{p}\|\mathbf{x}(I-A_\Delta^{-1}\mathbf{x}^{\top}\mathbf{x})\boldsymbol{\beta}\|^2+\|(I-\mathbf{x}A_\Delta^{-1}\mathbf{x}^{\top})\epsilon\|^2\\
    =&\lim_{p\to\infty}\frac{1}{p}[\Delta^2{\rm tr} ( A_\Delta^{-1} ) + \Delta(N-p)]\\
    =&\Delta T(\Delta) + \Delta(\alpha-1).
\end{align}
Substituting this result into the TAP functional defined in equation \eqref{tap_equation} concludes the proof.
\end{proof}






\subsection{Asymptotics of the Free Energy}
Similar to the work of \citet{Barbier2020}, we expect that $1 - q = \text{MSE}$, and that the true free energy can be determined by fixed-point equations. This formulation allows for a direct validation of its equivalence to the TAP representation.

For the Gaussian prior $P_0 \sim \mathcal{N}(0,1)$, the fourth moment is finite and the Single-Crossing Property holds (see Appendix \ref{section_B} for details). Thus, $P_0$ satisfies Assumptions II and III as proposed by \citet{reeves2019replica}.
The mutual information for the denoising model $i(\tilde{S};\tilde{Y})$ in \eqref{MI_denoising_model} can then be expressed in a simplified form as:
\begin{align}
i(\tilde{S};\tilde{Y})=
-\frac{1}{2}\ln(\frac{\Sigma^2}{\Sigma^2+1}).
\end{align}
This result facilitates the calculation of RS potential in \eqref{RS_potential}, which can be expressed as:
\begin{align}\label{RS_GS}
    &-i^{\mathrm{RS}}(E;\Delta) \notag\\=&\frac{E}{2\Sigma^2}- \frac{1}{2}\alpha \ln(1+E/(\alpha\Delta))+\frac{1}{2}\ln(\frac{\Sigma^2}{\Sigma^2+1}).
\end{align}

By setting its derivative to zero, the above equation can be further simplified to derive the fixed-point condition for $E$. Specifically, we obtain:
\begin{align}\label{fixed_point_condition}
    E=\frac{\alpha\Delta+E}{\alpha\Delta+E+\alpha}=\frac{\Sigma^2}{\Sigma^2+1},
\end{align}
which has a unique solution for $E\ge 0$.
Alternatively, this can be expressed as:
\begin{align}
E_\Delta=\frac{1-\alpha-\alpha\Delta+\sqrt{(\alpha\Delta+\alpha-1)^2+4\alpha\Delta}}{2}=T(\Delta).
    \label{closed_form_solution_E}
\end{align}
This result also suggests that the assumptions formulated by \citet{Barbier2020}, specifically that $i^{\mathrm{RS}}(E;\Delta)$ has at most three stationary points, as well as Assumption III proposed by \citet{reeves2019replica} are valid. Consequently, the RS formula is correctly established.

The R.H.S. of \eqref{theorem} can then be written as:
\begin{align}\label{RS_free_energy}
    &-\min_{E\ge 0} i^{\mathrm{RS}}(E;\Delta)-\frac{\alpha}{2} \notag\\=&-\frac{\alpha}{2}+\frac{1}{2}(1-E_\Delta)-\frac{1}{2}\alpha \ln(1+E_\Delta/(\alpha\Delta))+\frac{1}{2}\ln(E_\Delta).
\end{align}
Observing the equivalence between $E_\Delta$ and $1-q_\Delta$, we can readily derive the following conclusion. 


\begin{lemma}\label{lower}
Let $f_{\mathrm{TAP}}$ and $i^{\mathrm{RS}}(E;\Delta)$ be defined as \eqref{tap_equation} and \eqref{RS_GS}. Then the TAP free energy provides a lower bound matching the replica prediction with convergence in probability:
\begin{align}   \lim_{p\to\infty}\max_{\|a\|^2\le p}f_{\mathrm{TAP}}(a)
    \ge -\min_{E\ge 0} i^{\mathrm{RS}}(E;\Delta)-\frac{\alpha}{2}.
\end{align}
\end{lemma}
\begin{proof}
With $a_{\Delta}$ defined in \eqref{e_ridge}, by lemma \ref{lem4} and \eqref{RS_free_energy}, we have:
\begin{align}   &\lim_{p\to\infty}\max_{\|a\|^2\le p}f_{\mathrm{TAP}}(a)
\\\ge& 
\lim_{p\to\infty}f_{\mathrm{TAP}}(a_{\Delta})= -\min_{E\ge 0} i^{\mathrm{RS}}(E;\Delta)-\frac{\alpha}{2},
\end{align}  
with the convergence in probability.
\end{proof}

\subsection{Global Maximum}

The preceding analysis shows that $f=f_{\mathrm{TAP}}(1-E_\Delta)\le \max_q f_{\mathrm{TAP}}(q)$, and $q_\Delta=1-E_\Delta$ is a stationary point of $f_{\mathrm{TAP}}$.

Define 
\begin{align}
g_{\mathrm{TAP}}(q):=&\max_{\|a\|^2= pq}\{-\frac1{2\Delta p}\|y-Xa\|^2
\}\notag\\&-\frac{\alpha}{2}\ln(1+\frac{1-q}{\Delta\alpha})
+\frac1{2}\ln(1-q).
\label{e55}
\end{align}

The goal is to show that $q_\Delta=1-E_\Delta$ is (asymptotically) the global maximizer of $g_{\mathrm{TAP}}$.
Define
\begin{align}
g(q):=\max_{\|a\|^2=pq}\{-\frac1{2\Delta p}\|y-Xa\|^2\}-\frac{q}{2}
+C,
\label{e53}
\end{align}
where $C$ is a constant chosen such that $g(1-E_\Delta)=g_{\mathrm{TAP}}(1-E_\Delta)$,
that is,
\begin{align}
C:=\frac{1-E_\Delta}{2}
-\frac{\alpha}{2}\ln(1+\frac{E_\Delta}{\Delta\alpha})
+\frac1{2}\ln E_\Delta.
\label{e_c}
\end{align}

\begin{lemma}\label{lem7}
Let $g$ and $g_{\mathrm{TAP}}$ be defined as \eqref{e53} and \eqref{e55}. Then, for any $q\in[0,1]$, $
g_{\mathrm{TAP}}(q)\le g(q)    
$.
\end{lemma}
\begin{proof}
Define 
\begin{align}
h(q):= -q/2+C+\frac{\alpha}{2}\ln(1+\frac{1-q}{\Delta\alpha})
-\frac1{2}\ln(1-q),
\end{align}
with $C$ defined in \eqref{e_c}, so $h(1-E_0)=0$ by construction. 
The equation $h'(q)=0$ is equivalent to \eqref{fixed_point_condition} which has a unique solution for $q\in(0,1)$. 
Hence $h'(1-E_0)=0$ by the definition of $E_0$.
Furthermore, we observe that $h'(0)<0$ and $\lim_{q\uparrow 1}h'(q)=+\infty$.
Therefore $h'(q)>0$ for $q\in(1-E_0,1)$ and $h'(q)<0$ for $q\in(0,1-E_0)$.
Consequently, $g(q)-g_{\mathrm{TAP}}(q)=h(q)\ge0$ for $q\in(0,1)$.
\end{proof}

We are now ready to show the following:
\begin{lemma}\label{upper}
We have the matching upper bound:
\begin{align}   \lim_{p\to\infty}\max_{\|a\|^2\le p}f_{\mathrm{TAP}}(a)
\le -\min_{E\ge 0} i^{\mathrm{RS}}(E;\Delta)-\frac{\alpha}{2}.
\end{align}
\end{lemma}

\begin{proof}
We have
\begin{align}
\max_{\|a\|^2\le p}f_{\mathrm{TAP}}(a)
&\le\sup_{q\in[0,1]} g(q)
\\
&\le -\frac{1}{2\Delta p}\|\mathbf{y}-\mathbf{x}a_{\Delta}\|^2-\frac1{2p}\|a_{\Delta}\|^2+C.
\end{align}

Using the expressions for \( \|a_\Delta\|^2/p \) from \eqref{q_tr} and \( \|y - \mathbf{x}a_\Delta\|^2/p \) from \eqref{first_term}, we can evaluate the asymptotic value of the upper bound. Specifically, we have:
\begin{align}
    &\lim_{p\to\infty}\{-\frac{1}{2\Delta p}\|\mathbf{y}-\mathbf{x}a_{\Delta}\|^2-\frac1{2p}\|a_{\Delta}\|^2\}\notag\\=& -\frac{1}{2}(T(\Delta) + \alpha-1)-\frac{1}{2}(1-T(\Delta))
    =-\frac{\alpha}{2}.
\end{align}
This matches $C-\frac{\alpha}{2}=-i^{\mathrm{RS}}(E_\Delta;\Delta)-\frac{\alpha}{2}
= -\min_{E\ge 0} i^{\mathrm{RS}}(E;\Delta)-\frac{\alpha}{2}$. Therefore, the upper bound converges to the desired quantity, completing the proof.
\end{proof}

Combining the lower bound in Lemma \ref{lower} with the upper bound in Lemma \ref{upper}, we can conclude that
\begin{align}
    \lim_{p \to \infty} \max_{\|a\|^2 \leq p} f_{\mathrm{TAP}}(a) = -\min_{E\ge 0} i^{\mathrm{RS}}(E; \Delta) - \frac{\alpha}{2},
\end{align}
which establishes the asymptotic equivalence between the TAP free energy and the RS prediction for all $\Delta > 0$ under the uniform spherical prior. To conclude the proof of Theorem \ref{theorem_1} , it remains to
control the fluctuations of $F_p$ and show that convergence in expectation
implies convergence in probability. We address this in the following
subsection.

\subsection{Concentration of the Free Energy}

In the preceding analysis we established the convergence of the expected free
energy. To upgrade this to convergence
of the random free energy itself, it suffices to show that the fluctuations of $F_p$ vanish in the high-dimensional limit.

\begin{lemma}[Concentration of free energy]\label{lem_9}
Under the assumptions of Theorem \ref{theorem_1}, we have
\begin{align}
{\rm Var}(F_p) = O\left(\frac{1}{p}\right).
\end{align}
In particular,
\begin{align}
F_p - \mathbb{E}[F_p] \xrightarrow{P} 0.
\end{align}
\end{lemma}

The proof follows from Gaussian concentration applied to the Gaussian design matrix $X$ and noise vector $w$. The key observation is that $\log Z_p$ is Lipschitz in these random inputs with a Lipschitz constant of order $O(\sqrt{p})$ (See Appendix \ref{app_C} for details). Combining Lemma \ref{lower}, Lemma \ref{upper}, and Lemma \ref{lem_9}, we establish that
\begin{align}
F_p \xrightarrow{P}
\max_{\|a\|\leq \sqrt{p}} f_{\mathrm{TAP}}(a).
\end{align}
This completes the proof of Theorem \ref{theorem_1}.




\section{NUMERICAL SIMULATION}
We complement our theoretical results with numerical experiments that examine the finite-sample behavior of the TAP free energy in Bayesian linear regression. 
The purpose of this experiment is to provide empirical evidence for Corollary \ref{cor1}, which asserts that any near-maximizer of the TAP free energy lies close to the ridge regression solution. 
Accordingly, our simulations compare the maximizers of the TAP functional with the ridge estimator and examine how their distance behaves as the dimension grows.

We considered the model shown in the beginning of section 2. For each configuration $(\alpha,\Delta,p)$, we generate $\mathbf{x}$, $\beta$, and $\mathbf{y} = \mathbf{x}\beta + \sqrt{\Delta
}\mathbf{z}$. 
We then numerically optimize the TAP free energy functional $f_{\mathrm{TAP}}(a)$ with respect to $a \in \mathbb{R}^p$ under the constraint $\|a\|\le \sqrt{p}$. 
The optimization is performed by projected gradient ascent with a fixed step size and stopping tolerance. 
Because the TAP functional is non-convex, we use multiple restarts. 
Among all runs, we retain the maximum TAP value as our empirical estimate of the global maximizer.

To account for randomness, each experiment is repeated $R=30$ times with independent seeds. 
For each repetition we record (i) the TAP free energy at the estimated maximizer, 
(ii) the ridge regression free energy, and (iii) the squared distance: 
\begin{align}
       D = \frac{1}{p}\|a_{\text{TAP}} - a_{\text{ridge}}\|^2,
\end{align}
which quantifies the discrepancy between the TAP maximizer and the ridge solution. 
By Corollary \ref{cor1}, $D$ is predicted to vanish asymptotically; our simulations examine this prediction across system sizes ranging from $p=200$ to $p=1600$. 
\begin{figure}[h]
\centering
\includegraphics[width=0.45\textwidth]{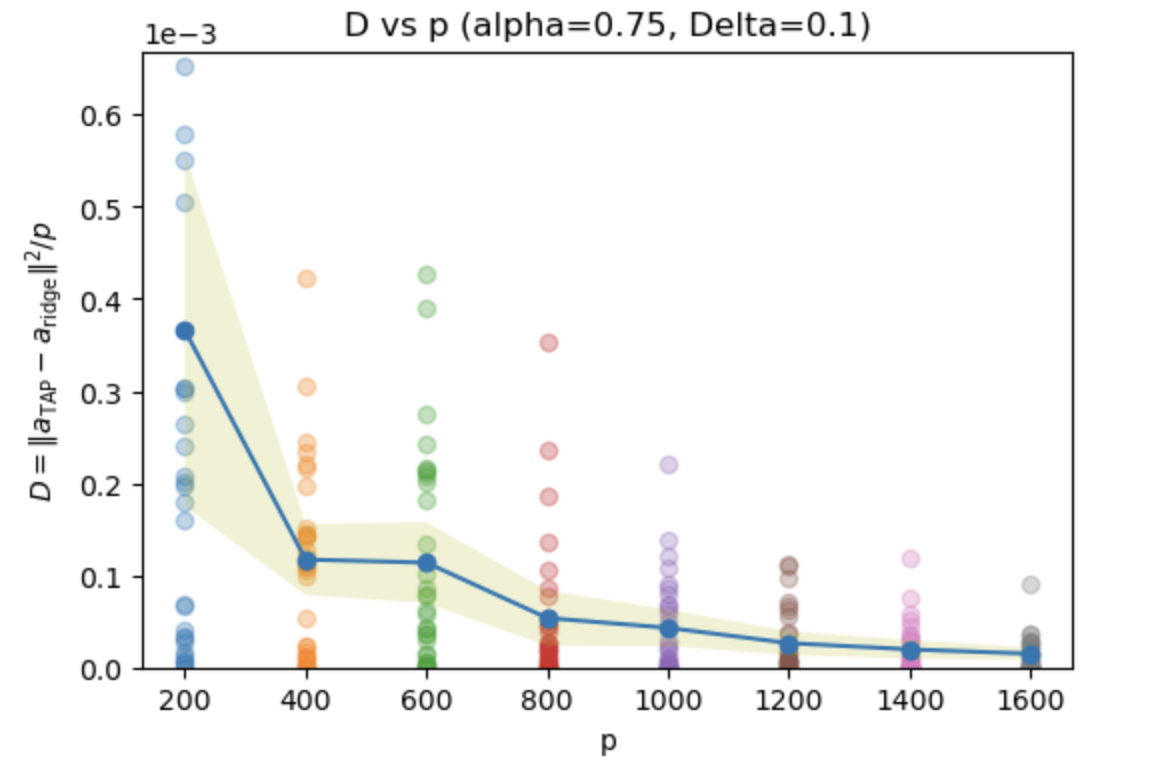}
\caption{Numerical Results Under $\Delta = 0.1$, $\alpha = 0.75$.}\label{fig1}
\end{figure}

Figure \ref{fig1} reports the squared distance $D = \|a_{\mathrm{TAP}}-a_{\mathrm{ridge}}\|^2/p$ 
between the TAP maximizer and the ridge regression estimator for a representative configuration $(\alpha=0.75,\Delta=0.1)$. 
The results clearly demonstrate that $D$ decreases as $p$ grows, in agreement with Corollary \ref{cor1}. 
Although finite-sample oscillations are visible for small $p$, the overall trend supports the theoretical prediction that 
maximizer of TAP concentrate around the ridge solution in the proportional high-dimensional regime. (see Appendix \ref{app_D} for additional experiments).

\section{CONCLUSION}
Our results establish that under suitable conditions, the TAP free energy admits a maximizer whose value converges to the true free energy. This work advances the theoretical understanding of the TAP free energy approximation in high-dimensional Bayesian inference. By establishing its equivalence to the RS formula, this study enhances the applicability of the TAP framework for high-dimensional Bayesian inference, offering a more accurate alternative to the widely used mean-field approximation.

Although we have proven that the TAP free energy converges in probability to the true free energy, the behavior of its maximizer under the uniform spherical prior remains to be explored.
In particular, while the TAP maximizer coincides with the posterior mean in the case of a Gaussian prior,
Corollary~\ref{cor1} suggests that the same holds under the spherical prior, though a rigorous justification requires additional work. Moreover, while our result shows that any near-maximizer of the TAP functional must be close to the ridge solution, this does not in itself provide the full geometric control established in the high-temperature analysis of \citet{qiu2023tap}. Our theorem resolves the question about the global maximizer of the TAP formula, but does not imply band-wise free energy control or pure-state structure. Establishing such a geometric characterization remains open, and we view it as a promising direction for future work.

Future directions include extending the TAP framework to more general prior distributions, including non-spherical and non-Gaussian settings, and assessing its applicability in more complex models. A natural next question is how to extend our results for spherical priors to the case of general i.i.d. priors, as conjectured by \citet{Krzakala2014}. 
To this end, the method-of-types approach from \cite{liu2022minoration,liu2023soft} for the free energy could provide an analogue to the geometric arguments used in spherical models \citep{qiu2023tap}.
Another important future direction is to develop and analyze algorithms for optimizing the TAP free energy in this non-convex regime. Understanding the geometry of the TAP landscape and designing robust optimization methods that can navigate potential spurious local maxima are essential steps toward practical inference in high-dimensional settings.

\subsubsection*{Acknowledgments}
This research was supported in part by NSF Grant DMS-2515510. 
The authors also thank Kevin Luo and Qiang Wu for helpful discussions.



%

\bibliography{ref.bib}
\newpage
\appendix
\begin{center}
\LARGE \textbf{Supplementary Materials}
\end{center}
\section{PROOF OF NON-CONCAVITY}\label{app_A}

In this section, we present a rigorous proof to illustrate that the TAP free energy functional may not be globally concave under the spherical prior.

Let
\begin{align}
h(q) := \frac{\alpha}{2} \log\left(1 + \frac{1 - q}{\Delta \alpha} \right) - \frac{1}{2} \log(1 - q). \label{eq:h_def}
\end{align}

The first and second derivative of $h$ can then be calculated as:
\begin{align}
    h'(q) = -\frac{\alpha}{2(\Delta\alpha+1-q)}+\frac{1}{2(1-q)},\\
    h''(q) =-\frac{\alpha}{2(\Delta\alpha+1-q)^2}+\frac{1}{2(1-q)^2}.
\end{align}
Thus, the Hessian of the TAP free energy satisfies:
\begin{align}
- p \cdot \nabla^2 f_{\mathrm{TAP}}(a) = \frac{4}{p} h''(q) \cdot a a^\top + 2 h'(q) I_p + \frac{1}{\Delta} \mathbf{x}^\top \mathbf{x}, \label{eq:hessian}
\end{align}
where $q := \|a\|^2/p$. Let $\lambda_{\min}$ denote the smallest eigenvalue of $X^\top X$, and $u_{\min}$ the corresponding eigenvector. From random matrix theory, the empirical spectral distribution of $X^\top X$ converges (as $p, n \to \infty$ with $n/p \to \alpha \in (1,\infty)$) to the Marchenko–Pastur law. In particular, the smallest eigenvalue $\lambda_{\min}$ converges almost surely to
\begin{align}
\lambda_{\min} \longrightarrow \left(1 - \sqrt{\alpha^{-1}} \right)^2.
\end{align}

Then, we consider a vector $a$ aligned with $u_{\min}$ of the form
\begin{align}
a = \frac{u_{\min}}{\|u_{\min}\|} \cdot \sqrt{p q}.
\end{align}

Then the second directional derivative becomes
\begin{align}
- a^\top \nabla^2 f_{\mathrm{TAP}}(a) a = q \left( 4q h''(q) + 2 h'(q) + \frac{\lambda_{\min}}{\Delta}\right). \label{eq:directional}
\end{align}

Dividing both sides by $q$, we obtain
\begin{align}
\frac{1}{q} \cdot \left( -a^\top \nabla^2 f_{\mathrm{TAP}}(a) a \right)
=& -\frac{2 \alpha q}{(\Delta \alpha + 1 - q)^2}
+ \frac{2q}{(1 - q)^2}
+ \frac{1}{1 - q}
- \frac{\alpha}{\Delta \alpha + 1 - q}
+ \frac{\lambda_{\min}}{\Delta}\\
=&-\frac{(\Delta\alpha+1+\ q)\alpha}{(\Delta \alpha + 1 - q)^2}
+ \frac{1+q}{(1 - q)^2}
+ \frac{\lambda_{\min}}{\Delta}\label{eq:second_deriv}
\end{align}

Letting $\Delta \alpha = k$, and recalling the asymptotic limit of $\lambda_{\min}$, as $p \to \infty$, the expression becomes:
\begin{align}
    -\frac{(k+1+\ q)\alpha}{(k + 1 - q)^2}
+ \frac{1+q}{(1 - q)^2}
+ \frac{\alpha+1-2\sqrt{\alpha}}{k}.\label{eq71}
\end{align}

To ensure this quantity becomes negative as $\alpha \to \infty$, it suffices to find $k > 0$ and $q \in (0, 1)$ such that:
\begin{align}
    \frac{1}{k}-\frac{k+1+\ q}{(k + 1 - q)^2}<0.
\end{align}

As an example, one can choose $k = 1$, $q = 0.5$, and set $\alpha = 10$, $\Delta = 0.1$, which yields:
\begin{align}
    -\frac{(k+1+\ q)\alpha}{(k + 1 - q)^2}
+ \frac{1+q}{(1 - q)^2}
+ \frac{\alpha+1-2\sqrt{\alpha}}{k} \approx -0.4357<0
\end{align}
 which implies that $- \nabla^2 f_{\mathrm{TAP}}(a)$ may have negative eigenvalues in the neighborhood of $a = \frac{u_{\min}}{\|u_{\min}\|} \cdot \sqrt{p/2}$. Therefore, the TAP free energy is not globally concave under the uniform spherical prior.

\section{VERIFICATION OF RS ASSUMPTIONS FOR THE GAUSSIAN PRIOR}\label{section_B}

In this section, we verify that the assumptions required for the validity of the RS formula, as formulated by \citet{reeves2019replica}, are satisfied under the standard Gaussian prior $P_0 = \mathcal{N}(0,1)$. 
We adopt the random linear estimation model in \eqref{mod_1}.

Define the normalized mutual information and MMSE functions:
\begin{align}
I_p(\alpha):=\frac{1}{p}\,I(\boldsymbol{\beta};\mathbf{y}\mid \mathbf{x}),\qquad
M_p(\alpha):=\frac{1}{p}\mathrm{mmse}(\boldsymbol{\beta}\mid \mathbf{y},\mathbf{x}).
\end{align}

Let $\beta \sim P_0$ and $z \sim \mathcal{N}(0,1)$. Consider the scalar Gaussian channel
\begin{align}
y = \sqrt{s}\beta + z.
\end{align}

We define the scalar functions:
\begin{align}
I_\beta(s) = I(\beta;\sqrt{s}\beta+z), \qquad \mathrm{mmse}_\beta(s) = \mathrm{mmse}(\beta\mid \sqrt{s}\beta+z).
\end{align}
The RS potential is then defined as:
\begin{align}
\mathcal{R}(\alpha,u) 
= I_\beta\Big(\frac{\alpha}{1+u}\Big) + \frac{\alpha}{2}\Big(\log(1+u) - \frac{u}{1+u}\Big),
\end{align}

Using the I-MMSE identity and stationarity of $\mathcal{R}$, one obtains the RS fixed-point equation
\begin{align}
u = \mathrm{mmse}_\beta\Big(\frac{\alpha}{1+u}\Big).
\label{eq78}
\end{align}
It is convenient to view \eqref{eq78} as the intersection of two curves in the $(u,\alpha)$-plane. Define
\begin{align}
\alpha_{\mathrm{FP}}(u) 
:= (1+u)\mathrm{mmse}_\beta^{-1}(u),
\qquad
\alpha_{\mathrm{RS}}(u)
:= \inf\{\alpha\ge 0:\ u\in\arg\min_{v\ge 0}\mathcal{R}(\alpha,v)\}.
\end{align}
Equivalently, $\alpha_{\mathrm{RS}}(u)$ is characterized by the vanishing first variation of $\mathcal{R}$ at $u=v$.

The single-crossing property requires that the two curves $\alpha_{\mathrm{FP}}(u)$ and $\alpha_{\mathrm{RS}}(u)$ intersect at most once for $u\in(0,1)$. Under this property the following lemma holds:

\begin{lemma}{\rm{\citep{reeves2019replica}}}
Consider the random linear model in \eqref{mod_1}
where $\mathbf{x}\in\mathbb{R}^{n\times p}$ has i.i.d.\ entries $\mathbf{x}_{ij}\sim \mathcal N(0,1/p)$, the noise $Z\sim\mathcal N(0,I_m)$ is independent of $\mathbf{x}$, and the signal $\boldsymbol{\beta}=(\beta_1,\dots,\beta_p)$ has i.i.d.\ entries $\beta_i\sim P_0$ with $\mathbb{E}[\beta^4]<\infty$.

Assume $P_0$ satisfies the single-crossing property. Then, as $p\to\infty$,
\begin{align}
I_p(\alpha)\longrightarrow I_{\mathrm{RS}}(\alpha)
&:=\min_{u\ge 0}\mathcal R(\alpha,u),\\
M_p(\alpha)\longrightarrow M_{\mathrm{RS}}(\alpha)
&\in \arg\min_{u\ge 0}\,\mathcal R(\alpha,u).
\end{align}

\end{lemma}

Now we consider the special case $P_0=\mathcal{N}(0,1)$. We have $\mathbb{E}\beta^4=3<\infty$. For the scalar channel, $\mathrm{mmse}_\beta(s)=\frac{1}{1+s}$ and $I_\beta'(s)=\tfrac{1}{2}\mathrm{mmse}_\beta(s)$. Since $\mathrm{mmse}_\beta(s)=1/(1+s)$, its inverse on $(0,1)$ is $\mathrm{mmse}_\beta^{-1}(z)=(1-z)/z$. Hence
  \begin{align}
  \alpha_{\mathrm{FP}}(u)=(1+u)\frac{1-u}{u}=\frac{1-u^2}{u}.
  \end{align}
  The RS stationarity condition yields the same parametric relation between $\alpha$ and $u$:
  \begin{align}
  u = \frac{1}{1+\alpha/(1+u)}
  \quad\Longleftrightarrow\quad
  \alpha = \frac{1-u^2}{u},
  \end{align}
  so $\alpha_{\mathrm{RS}}(u)\equiv \alpha_{\mathrm{FP}}(u)$ on $(0,1)$.
Since $\alpha_{\mathrm{RS}}\equiv\alpha_{\mathrm{FP}}$ pointwise, the two curves coincide and in particular cross at most once. The single crossing property holds trivially.

All assumptions required by the framework of \citet{reeves2019replica} are satisfied for $P_\beta=\mathcal{N}(0,1)$. Consequently, along $n,p\to\infty$ with $n/p\to\alpha$,
\begin{align}
I_p(\alpha)\ \longrightarrow\ I_{\mathrm{RS}}(\alpha)=\min_{z\ge 0}\,\mathcal{R}(\alpha,u).\label{eq84}
\end{align}
where $u_\star(\alpha)$ is the unique minimizer of $\mathcal{R}(\alpha,\cdot)$ and solves $u=\mathrm{mmse}_\beta(\alpha/(1+u))$. And from \eqref{eq84}, we can easily obtain the convergence we want in \eqref{e_f}.

\section{PROOF OF LEMMA \ref{lem_9}}\label{app_C}

In this section we present the Concentration of $F_p$. Recall that, conditional on $X$,
\begin{align}
\mathcal{Z}_p(X,y) = (2\pi)^{-n/2}\,\det(A_\Delta)^{-1/2}\,
\exp\!\Big(-\tfrac12\,y^\top A_\Delta^{-1}y\Big),
\end{align}
where $A_\Delta=\Delta I_n+ XX^\top$.
Define
\begin{align}
f(X,y):=\frac1p\log \mathcal{Z}_p(X,y)
= -\frac{1}{2p}\log\det A_\Delta-\frac{1}{2p}\,y^\top A_\Delta^{-1}y-\frac{n}{2p}\log(2\pi).
\end{align}
We use the law of total variance:
\begin{align}
{\rm Var}\big(F_p\big) = \mathbb{E}\big[{\rm Var}(f\mid X)\big] + {\rm Var}\big(\mathbb{E}[f\mid X]\big).
\end{align}

Conditioning on $X$, we have $y\mid X\sim\mathcal N(0,A_\Delta)$ and
\begin{align}
\nabla_y f(X,y) = -\frac{1}{p}A_\Delta^{-1}y.
\end{align}
The Poincar\'e inequality with covariance $A_\Delta$ states that for smooth
$g:R^n\to R$ and $Y\sim\mathcal N(0,A_\Delta)$,
${\rm Var}(g(Y))\le \mathbb{E}\big[\nabla g(Y)^\top A_\Delta\nabla g(Y)\big]$.
Applying this with $g(y)=f(X,y)$ gives
\begin{align}
{\rm Var}(f\mid X)
\le \mathbb{E}\left[\frac{1}{p^2}y^\top A_\Delta^{-1}A_\Delta A_\Delta^{-1}y \Big|X\right]
= \frac{1}{p^2}\mathbb{E}\left[y^\top A_\Delta^{-1}y \Big|X\right]
= \frac{1}{p^2}{\rm tr}(A_\Delta^{-1}A_\Delta)
= \frac{n}{p^2}.
\end{align}
Taking expectation over $X$ yields
\begin{align}
\mathbb{E}[{\rm Var}(f\mid X)] \le\frac{n}{p^2} = \frac{\alpha}{p}.\label{lem_9_p1}
\end{align}

Then, we compute $\mathbb{E}[f\mid X]$. Since $\mathbb{E}[y^\top A_\Delta^{-1}y\mid X]={\rm tr}(A_\Delta^{-1}\mathbb{E}[yy^\top\mid X])={\rm tr}(A_\Delta^{-1}A_\Delta)=n$,
\begin{align}
\mathbb{E}[f\mid X] = -\frac{1}{2p}\log\det A_\Delta- \frac{n}{2p} - \frac{n}{2p}\log(2\pi),
\end{align}
only the $\log\det A_\Delta$ term depends on $X$. Matrix calculus gives
$\partial(\log\det A_\Delta)/\partial A_\Delta = A_\Delta^{-1}$ and
\begin{align}
dA_\Delta = (dXX^\top + XdX^\top).
\end{align}
By the chain rule and cyclicity of trace,
\begin{align}
d\mathbb{E}[f\mid X] = -\frac{1}{2p}\langle A_\Delta^{-1},dA_\Delta\rangle
= -\frac{\tau^2}{2p}{\rm tr}\big(A_\Delta^{-1}dX X^\top + A_\Delta^{-1}X dX^\top\big)
= -\frac{1}{p}{\rm tr}\big(A_\Delta^{-1}X dX^\top\big).
\end{align}
Hence the gradient is
\begin{align}
\nabla_X \mathbb{E}[f\mid X] = -\frac{1}{2p}A_\Delta^{-1}X.
\end{align}
Therefore
\begin{align}
\|\nabla_X \mathbb{E}[f\mid X]\|_F^2
= \frac{1}{4p^2}{\rm tr}\big(X^\top A_\Delta^{-2} X\big)
 \le \frac{1}{4p^2}\|A_\Delta^{-2}\|\|X\|_F^2
\le \frac{1}{4p^2\Delta^2}\|X\|_F^2,
\end{align}
since $A_\Delta\succeq \Delta I_n$ implies $\|A_\Delta^{-1}\|\le 1/\Delta$ and thus $\|A_\Delta^{-2}\|\le 1/\Delta^2$.

Now apply the Poincar\'e inequality to $X$:
\begin{align}
{\rm Var}(\mathbb{E}[f\mid X])
\le \frac{1}{p}\mathbb{E}\|\nabla_X \mathbb{E}[f\mid X]\|_F^2
\le \frac{1}{p}\cdot \frac{1}{4p^2\Delta^2}\mathbb{E}\|X\|_F^2
= \frac{1}{4p^3\Delta^2}\, \mathbb{E}\Big[\sum_{i,j}X_{ij}^2\Big]
= \frac{1}{4p^3\Delta^2} n
= \frac{1}{4\Delta^2}\frac{\alpha}{p^2}.\label{lem_9_p2}
\end{align}

Combing the bounds in \eqref{lem_9_p1} and \eqref{lem_9_p2},
\begin{align}
{\rm Var}(F_p) \le \frac{\alpha}{p} + \frac{1}{4\Delta^2}\frac{\alpha}{p^2}.
\end{align}
 Chebyshev’s inequality then implies
$F_p-\mathbb{E}[F_p]\to 0$ in probability.

\section{ADDITIONAL NUMERICAL STUDIES}\label{app_D}

In this section, we provide extended experiments that go beyond those presented in Section~5. Specifically, we study a grid of parameters with sampling ratios $\alpha = 0.75, 1.0,2.0,$ and noise levels $\Delta = 0.1,0.5,1.0$. For each $(\alpha, \Delta)$ pair, we evaluate the discrepancy
\begin{align}
D = \frac{1}{p}\|a_{\mathrm{TAP}} - a_{\mathrm{ridge}}\|^2,
\end{align}
as a function of the system size $p$.

In all settings tested, $D$ decreases with $p$, consistent with the theoretical prediction. The small oscillations observed at intermediate values of $p$ are consistent with finite-size effects and are expected to vanish at rate $O(1/p)$.  

The complete set of nine plots, covering all combinations of $(\alpha, \Delta)$, is shown in Figure~\ref{fig2}.

\newpage
\begin{figure}[h]
\centering
\includegraphics[width=\textwidth]{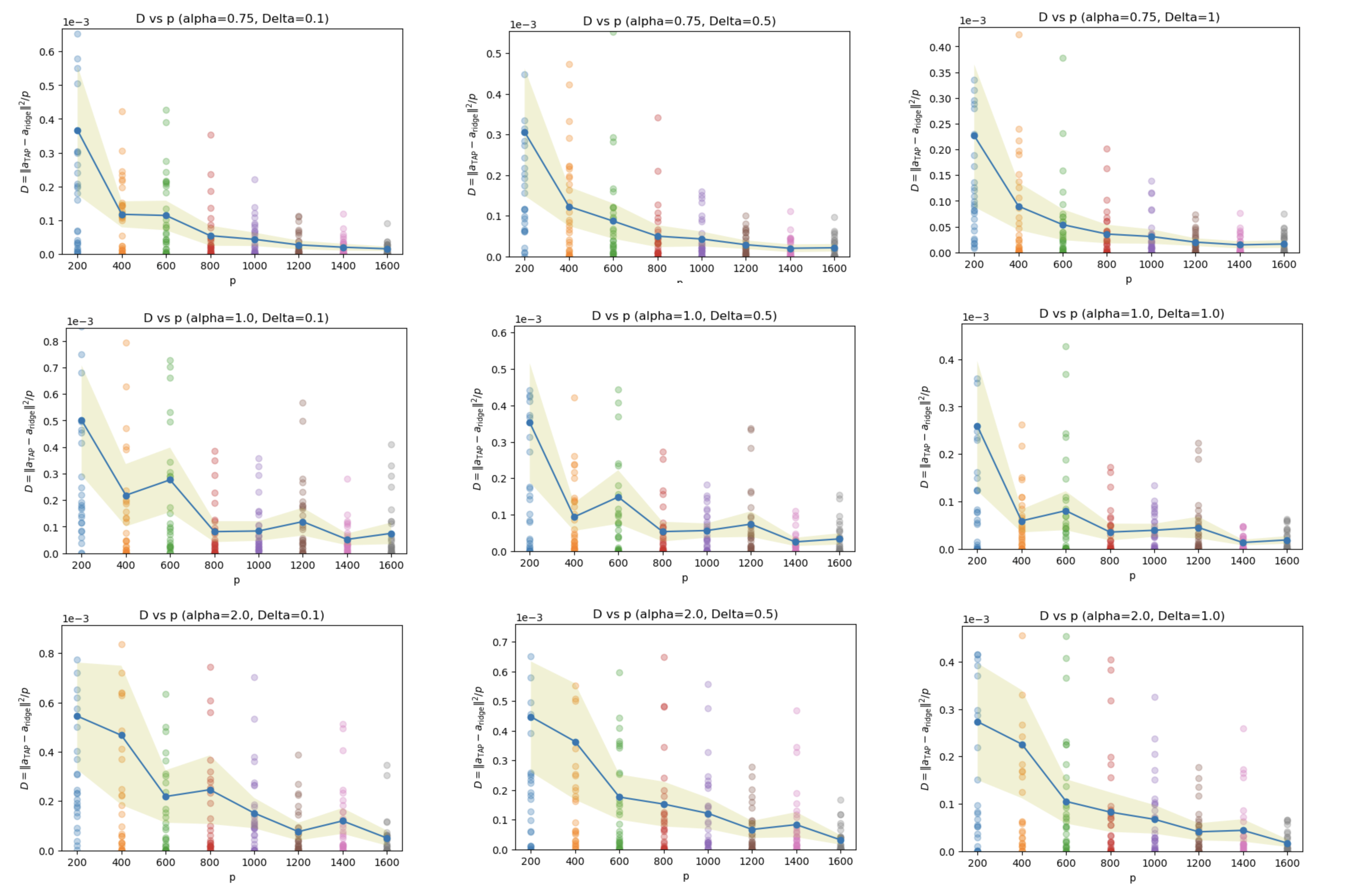}
\caption{Complete Collection of Results, Displaying the Discrepancy $D$ across All Combinations of $(\alpha,\Delta)$.}
\label{fig2}
\end{figure}

\end{document}